\documentclass[11pt]{amsart}
\usepackage{latexsym, amsmath,amssymb}
\usepackage[normalem]{ulem}
\usepackage{hyperref}
\usepackage{xcolor}
\hypersetup{
    colorlinks,
    linkcolor={red!50!black},
    citecolor={blue!50!black},
    urlcolor={blue!80!black}
}

\usepackage{graphicx}
\usepackage{tikz}

 \numberwithin{equation}{section}

\def\XXint#1#2#3{{\setbox0=\hbox{$#1{#2#3}{%
\int}$ }
\vcenter{\hbox{$#2#3$ }}\kern-.6\wd0}}

\renewcommand{\epsilon}{\varepsilon}
\newtheorem{theorem}{Theorem}

\newtheorem{lemma}[theorem]{Lemma}
\newtheorem{corr}[theorem]{Corollary}

\newtheorem{proposition}[theorem]{Proposition}
\newtheorem{deff}[theorem]{Definition}
\newtheorem{remark}[theorem]{Remark}
\newcommand{\bth}{\begin{theorem}}
\newcommand{\ble}{\begin{lemma}}
\newcommand{\bcor}{\begin{corr}}

\newcommand{\bdeff}{\begin{deff}}

\newcommand{\bprop}{\begin{proposition}}
\newcommand{\ele}{\end{lemma}}
\newcommand{\ecor}{\end{corr}}
\newcommand{\edeff}{\end{deff}}

\numberwithin{theorem}{section}

\newcommand{\eprop}{\end{proposition}}

\renewcommand{\Pi}{\varPi}

\renewcommand{\epsilon}{\varepsilon}

\newcommand{\R}{{\mathbb R}}
\newcommand{\Z}{{\mathbb Z}}

\begin{document}

\title[Pointwise monotonicity of heat kernels]
{Pointwise monotonicity of heat kernels}


\date{\today}

\author[D. Alonso-Or\'an]{D. Alonso-Or\'an}
\address{Instituto de Ciencias Matem\'aticas (CSIC-UAM-UC3M-UCM) -- Departamento de Matem\'aticas (Universidad Aut\'onoma de Madrid), 28049 Madrid, Spain} 
\email{diego.alonso@icmat.es}
\author[F. Chamizo]{F. Chamizo}
\address{Instituto de Ciencias Matem\'aticas (CSIC-UAM-UC3M-UCM) -- Departamento de Matem\'aticas (Universidad Aut\'onoma de Madrid), 28049 Madrid, Spain} 
\email{fernando.chamizo@uam.es}
\author[A. D. Mart\'inez]{\'A. D. Mart\'inez}
\address{Instituto de Ciencias Matem\'aticas (CSIC-UAM-UC3M-UCM) -- Departamento de Matem\'aticas (Universidad Aut\'onoma de Madrid), 28049 Madrid, Spain} 
\email{angel.martinez@icmat.es}
\author[A. Mas]{A. Mas}
\address{Departament de Matem\`atiques i Inform\`atica, Universitat de Barcelona, Gran {V\'ia} 585, 08007 Barcelona, Spain}
\email{albert.mas@ub.edu}

\begin{abstract}
In this paper the authors present a proof of a pointwise radial monotonicity property of heat kernels that is shared by the euclidean spaces, spheres and hyperbolic spaces. The main result deals with monotonicity from special points on revolution hypersurfaces from which the aforementioned are deduced. The proof relies on a non straightforward but elementary application of the parabolic maximum principle.
\end{abstract}

\maketitle

\section{Introduction}

The heat equation is one of the quintessentials among mathematical models for physical phenomena. Over the years, several properties of this equation had been studied  from different points of view, including for instance: probabilistic, geometric and physical. In this paper we will focus on the fundamental solution of the heat equation, namely, the heat kernel. For small times a parametrix is well known.  It allows to confirm the heuristic fact that it should behave like the euclidean heat kernel for small time scales. 

In this work we give a rigorous proof of the following intuitively true result and some generalizations to manifolds with symmetries.

\begin{theorem}\label{decreasing}
Let $M$ be $\mathbb{R}^n$, $\mathbb{S}^n$ or $\mathbb{H}^n$. Then, for any fixed $x\in M$ and time $t\in(0,\infty)$, the heat kernel $G(x,y,t)$ is a decreasing function of the geodesic distance $d(x,y)$.
\end{theorem}

It is easy to check that Theorem \ref{decreasing} holds in the case of $M$ {being} $\R^n$ since Fourier analysis provides an explicit expression of the kernel, namely
\[G(x,y,t)=\frac{1}{(4\pi t)^{n/2}}\exp\left(-\frac{|x-y|^2}{4t}\right).\]
Explicit expressions like the above are rare. An exception would be hyperbolic spaces, 
e.g. for the hyperbolic plane one gets the following
\[G(x,y,t)=\frac{\sqrt{2}e^{-t/4}}{(4\pi t)^{3/2}}\int_{d(x,y)}^{\infty}\frac{\beta e^{-\beta^2/4t}}{\sqrt{\cosh(\beta)-\cosh(d(x,y))}}\,d\beta.\]
In general, distinguishing odd and even dimensional cases, rather involved formuli are available.(cf. \cite{Ch}). 

The sphere does not share this good fortune. Of course, one may argue against this provocative statement that spectral expansions are available. The best results  available in the literature deal with the one, two and three dimensional cases (cf. \cite{Andersson}). The proof is quite elaborated, based on specific estimates using spherical harmonics and does not seem to generalize in a straightforward way to higher dimensions. Recently, Nowak, Sj{\"o}gren and Szarek \cite{NSS} provided sharp estimates for the heat kernel on the sphere $\mathbb{S}^n$ that imply Theorem \ref{decreasing} in that specific case. Their proof is {fairly} technical and relies on certain recurrence relations for the heat kernels of spheres of different dimensions. Our neat proof, nevertheless, is built in a delicate application of the parabolic maximum principle. The same arguments also apply to more general situations described below, of which Theorem~\ref{decreasing} is a rather beautiful particular case.

\begin{theorem}\label{revolution}
Let $M\subseteq\mathbb{R}^n$ be a smooth, compact and connected hypersurface of revolution around the $x_n$ axis. If $x$ is a point of intersection of $M$ and the $x_n$ axis, then the associated heat kernel $G(x,y,t)$ decreases as a function of the geodesic distance $d(x,y)$ for any fixed $t>0$.
\end{theorem}

The same proof covers the noncompact situation even in an intrinsic geometric setting beyond hypersurfaces of $ \R^n$. In connection with this, recall that a celebrated theorem by Hilbert  states that a complete regular surface of constant negative curvature, like~$\mathbb{H}^2$, cannot be  isometrically immersed in $\R^3$. 

Theorem \ref{revolution} can be proven even in more general settings. Let us introduce some definitions before stating this. In a complete riemannian manifold~$M$ of dimension $n$ a point~$p$ is called a pole if its cut locus is empty. A manifold is said to be spherically symmetric (around~$p$) if its metric has the form $d\rho^2+A^2(\rho)d\sigma^2$ where $d\sigma$ is the line element of $\mathbb{S}^{n-1}$. In other words, the rotations around the origin in $T_p(M)$ become isometries of $M$ under the exponential map.\footnote{This kind of manifolds are called  ``model manifolds with $R_0=\infty$'' in \cite{grigoryan} although we shall not employ this name here.} Finally, we need some control on the volume of the surface of balls in order to apply the maximum principle.

\begin{theorem}\label{model}
 Let $M$ be a complete spherically symmetric manifold with bounded curvature. Let $S$ be the volume of the surface ball of radius $\rho$ in $M$ centered on a pole $p$ and suppose that $\frac{\partial^{2}}{\partial\rho ^{2}}\log (S(\rho))$ is bounded from above. Then the heat kernel $G(p,y,t)$ based on the pole $p$, satisfies $\int_M G(p,y,t)\; dy=1$ and  it is a decreasing function of the geodesic distance $d(p,y)$ for any fixed $t>0$. 
\end{theorem}

Our last result tackels the Dirichlet and Neumann heat kernels, $G_D$ and $G_N$ respectively, of a smooth hypersurface of revolution $M\subseteq\R^n$ with boundary. Recall that $G_D$ is the fundamental solution of the Dirichlet heat operator $\partial_t-\Delta_M$ 
(where we denoted by $-\Delta_M$ the positive Laplace-Beltrami operator on $M$) with the Dirichlet boundary condition, that is, for a fixed $x\in M$, $G_D(x,y,t)$ is the function in $(y,t)$ satisfying
\begin{equation*}
\begin{cases}
(\partial_t-\Delta_M)G_D(x,y,t)=0, &(y,t)\in M\times(0,+\infty),\\
G_D(x,y,t)=0, &(y,t)\in\partial M\times(0,+\infty),\\
G_D(x,y,0)=\delta_x(y), & y\in M.
\end{cases}
\end{equation*}
Similarly, for $G_N$ with the Neumann boundary condition, that is, for a fixed $x\in M$, $G_N(x,y,t)$ is the function in $(y,t)$ satisfying
\begin{equation*}
\begin{cases}
(\partial_t-\Delta_M)G_N(x,y,t)=0, &(y,t)\in M\times(0,+\infty),\\
\frac{\partial}{\partial n}\,G_N(x,y,t)=0, &(y,t)\in\partial M\times(0,+\infty),\\
G_N(x,y,0)=\delta_x(y), & y\in M.
\end{cases}
\end{equation*}

\begin{theorem}\label{neumann}
Let $M\subseteq\mathbb{R}^n$ be smooth and connected hypersurface of revolution around the $x_n$ axis with boundary $\partial M\neq\emptyset$. 
\begin{itemize}
\item[(i)] If $x$ is a point of intersection of the relative interior of $M$ and the $x_n$ axis, then the associated heat kernel with Dirichlet boundary condition $G_D(x,y,t)$ decreases as a function of the geodesic distance $d(x,y)$ for any fixed time $t>0$.
\item[(ii)] If $x$ is a point of intersection of the relative interior of $M$ and the $x_n$ axis, then the associated heat kernel with Neumann boundary condition $G_N(x,y,t)$ decreases as a function of the geodesic distance $d(x,y)$ for any fixed time $t>0$.
\end{itemize}
\end{theorem}

An analogous result in the case of the Dirichlet heat kernel on a geodesic ball inside a $n$-dimensional simply connected space form of constant sectional curvature $k$ was already known (cf. \cite[\S8.3]{Ch}).

The paper is organized along the following lines: in the next section we present the proofs of the obtained results where we demonstrate  Theorem~\ref{revolution}, sketch the proof of Theorem~\ref{model} and infer as a consequence Theorem~\ref{decreasing}. The third section is devoted to show some applications regarding inequalities of orthogonal polynomials and pointwise properties of the fractional Laplace-Beltrami operator. The article ends raising a natural question that bonds the decreasing properties of the heat kernel from a point with its cut locus.

\textsc{Note added in proof:} D. Nix has recently informed us that Theorem \ref{decreasing} has already been proved by Cheeger and Yau \cite{CY}. Their interest originates from the possibility of comparing the fundamental solutions to the heat equation on general manifolds with heat kernels of model manifolds. Their proof is quite similar in spirit to that of \cite{Ch}. Nevertheless, we could highlight that our approach gives a more direct derivation from the parabolic maximum principle. It worth also mentioning that we include some applications.

\section{Proof of theorem \ref{revolution}}

Thanks to the symmetry around the $x_n$ axis, the heat kernel $G(x,y,t)$ is a function of the geodesic distance from $x$ to $y$, $\rho=d(x,y)$, and the time variable $t$ only, that is, $G(x,y,t)\equiv G(\rho,t)$. This justifies expressing the heat equation on any given function
$f:[0,L]\times(0,+\infty)\to\R$ in geodesic coordinates as \cite[\S 3.2]{grigoryan}
\begin{equation}\label{heat regularized}
\frac{\partial}{\partial t}f(\rho,t)=\frac{\partial^2}{\partial \rho^2}f(\rho,t)+\frac{\partial}{\partial \rho}\log(S(\rho))\frac{\partial}{\partial\rho}f(\rho,t),
\end{equation}
where $L$ denotes the geodesic distance from $x$ to its antipodal point and 
{$S(\rho)$ denotes the volume of the surface of the ball of radius $\rho$ in $M$ centered at $x$}. 

Instead of working directly with the heat kernel, let us regularize it employing a family of functions $\{\chi_{\epsilon}(x,\cdot)\}_{\epsilon>0}$, that satisfy the following properties: smoothness, radial ({i.e.} $\chi_{\epsilon}(x,y)\equiv \chi_{\epsilon}(\rho)$), decreasing from $x$, integrate one, and concentrate around the fixed point $x\in M$ as $\epsilon$ tends to zero. Notice that this family is, in particular, an approximation of the identity. Let us now introduce the aforementioned regularization of the heat kernel
\[
F_{\epsilon}(x,y,t)=\int_{M}\chi_{\epsilon}(x,z)G(z,y,t)d\sigma(z),
\]
where $d\sigma$ denotes the volume form on $M$. Note that $F_{\epsilon}$ satisfies the heat equation {in} the variables $(y,t)$ and initial condition 
$\chi_{\epsilon}(x,\cdot)$. This solution is smooth and, due to the symmetry, 
$F_{\epsilon}(x,y,t):=F_\epsilon(d(x,y),t)$ {can be supposed to depend on the radial coordinate $\rho=d(x,y)$model emanating from $x$. Furthermore, its first derivative} with respect to $\rho$ {vanishes} at $x$ and its antipodal point. 

We will apply a parabolic maximum principle in the coordinates $(\rho,t)\in[0,L]\times[0,\infty)$ to prove that
\begin{equation}\label{heat regularized 3}
D_{\epsilon}(\rho,t):=\frac{\partial}{\partial\rho}F_{\epsilon}(\rho,t)\leq0.
\end{equation}
This means that $F_{\epsilon}$ is a smooth radially decreasing function for all 
 $\epsilon$. The limit of such family is radially decreasing. {Hence}, the argument finishes observing that $\chi_{\epsilon}(x,\cdot)$ is an approximation of the identity around the fixed point $x$ {and thus} $F_{\epsilon}(x,y,t)\to G(x,y,t)$ as $\epsilon\to0$. It only remains to prove \eqref{heat regularized 3}.
  
Note that $D_{\epsilon}(0,t)=D_{\epsilon}(L,t)=0$ for all $t>0$ and  
$D_{\epsilon}(\rho,0)\leq0$ for $t=0$ by construction. Using that 
$F_\epsilon$ satisfies \eqref{heat regularized}, by differentiating in $\rho$ we get
\begin{equation}\label{heat regularized 2}
\begin{split}
\frac{\partial}{\partial t}D_{\epsilon}(\rho,t)
&=\frac{\partial^2}{\partial \rho^2}D_{\epsilon}(\rho,t)+\frac{\partial}{\partial \rho}\log(S(\rho))\frac{\partial}{\partial\rho}D_{\epsilon}(\rho,t)\\
&\quad+\frac{\partial^2}{\partial \rho^2}\log(S(\rho))D_{\epsilon}(\rho,t).
\end{split}
\end{equation}
Since $D_\epsilon$ is continuous in $[0,L]\times[0,+\infty)$, given $T>0$ there exists
$(\rho_0,t_0)\in [0,L]\times[0,T]$  such that
\begin{equation}
D_{\epsilon}(\rho_0,t_0)=\sup_{(\rho,t)\in [0,L]\times[0,T]}D_{\epsilon}(\rho,t).
\end{equation}
If  $(\rho_0,t_0)\in ([0,L]\times\{0\})\cup(\{0,L\}\times[0,T])$ we directly get that 
\eqref{heat regularized 3} holds for all $(\rho,t)\in[0,L]\times[0,T]$.
On the contrary, if
$(\rho_0,t_0)\in (0,L)\times(0,T]$ then
\begin{equation}\label{heat regularized 4}
\frac{\partial}{\partial t}D_{\epsilon}(\rho_0,t_0)\geq0,\quad
\frac{\partial^2}{\partial \rho^2}D_{\epsilon}(\rho_0,t_0)\leq0,
\quad \frac{\partial}{\partial\rho}D_{\epsilon}(\rho_0,t_0)=0.
\end{equation}
Assume that $\frac{\partial^2}{\partial\rho^2}\log(S(\rho))<0$ in $(0,L)$. Then \eqref{heat regularized 4} contradicts \eqref{heat regularized 2} unless $D_{\epsilon}(\rho_0,t_0)\leq0$. Hence
\eqref{heat regularized 3} holds for $(\rho_0,t_0)$ and thus for all 
$(\rho,t)\in[0,L]\times[0,T]$. We finally get the desired estimate letting $T$ go to infinity.

For the general case where the coefficient of the zero order term in \eqref{heat regularized 2} is not strictly negative in $(0,L)$, we can still apply the parabolic maximum principle whenever the coefficient is bounded from above (cf. \cite{PW} or \cite{E}, p.\! 426). More precisely, near $x$ we see that $\frac{\partial^2}{\partial\rho^2}\log(S(\rho))$ behaves like in the euclidean case. This provides the approximation $-(n-2)\rho^{-2}$ for $\rho$ near zero, and the rest can be bounded by continuity and compactness. Therefore, $\frac{\partial^2}{\partial\rho^2}\log(S(\rho))$ is bounded from above if $M$ is compact. Then, the change 
$$v(\rho,t)=e^{\lambda t} D_\epsilon(\rho,t)$$ for a suitable 
$\lambda$ rephrases \eqref{heat regularized 2} to an equation on $v$ for which the zero order term has a strictly negative coefficient, and the previous argument applies. We omit {further} details.

\begin{remark}{\em Inspection of the proof shows that all we need is $\frac{\partial^2}{\partial\rho^2}\log(S(\rho))$ to be bounded from above and decay at infinity which was granted in the previous case due to the compactness (by continuity and discreteness of the spectrum). Next, we will demonstrate how to justify the decay condition for the noncompact case.}
\end{remark}

\begin{proof}[Proof of Theorem \ref{model}]
The proof can be mimicked from the previous one. Since the boundedness of $\frac{\partial^2}{\partial\rho^2}\log(S(\rho))$ is part of the hypothesis, we just need to provide an exponential decay of the radial derivative of $G(x,\cdot,t)$. Under the geometric hypothesis in the statement Cheng, Li, and Yau have shown already that, for $t\in[0,T)$,
\[|\nabla G(x,y,t)|\leq C_T t^{(n+1)/2}e^{-cd(x,y)^2/t}\]
for some constants $c,\, C_T>0$ independent on the distance $d(x,y)$ (cf. Theorem 6, \cite[\S4]{CheLiYau} p. 1055). This is enough to ensure the decay at infinity to close the argument. Observe that one may use the maximum principle on a space-time box $B(r)\times[0,T]$, then the maxima should be achieved in the boundary $t=0$. Indeed, the heat equation prevents a maximum to be achieved in $t=T$; on the other hand  one can extend the box to be a band by letting $r$ tend to infinity (i.e. $M\times [0,T]$) since the boundary terms tend to zero (due to the exponential decay at infinity model the aformentioned bound provides).
\end{proof}\vspace{0.5cm}

\begin{remark}\label{rmk 2}{\em 
The proof of Theorem \ref{neumann} is completely analogous to the one of Theorem \ref{revolution}. Indeed, for Dirichlet heat kernels one observes that $G_D(x,y,t)\geq 0$ in the interior and it vanishes at the boundary. This shows that its radial derivative is less than or equal to zero at the boundary, the rest of the proof is analogous. On the other hand, for the Neumann heat kernel the boundary replaces the antipodal point and the Neumann condition on the boundary replaces the radial derivative that vanishes at the antipodal point.}
\end{remark}

\section{Some applications}

The fractional Laplace-Beltrami operator on manifolds provides diffusions connected with a family of stochastic processes known as L\'evy flights. It has remarkable properties, some of which can be deduced by subordination to the heat kernel due to the well known formula
\begin{equation}\label{int:lapla:bel}
(-\Delta_g)^{\alpha}f(x)=\dfrac{1}{\Gamma(-\alpha)}\int_0^{\infty}\left(f(x)-e^{-t\Delta_g}f(x)\right)\frac{dt}{t^{1+\alpha}},
\end{equation}
for $\alpha\in(0,1)$ and smooth $f$. Thanks to our previous results, we get the following
\begin{proposition}
Let $\alpha\in(0,1)$. There is a positive function of the distance $k_{n,\alpha}(d(x,y))$ such that $k_{n,\alpha}(0)=1$, $d(x,y)^{-n-\alpha}k_{n,\alpha}(d(x,y))$ is decreasing for $d(x,y)\in(0,\pi)$, and
\begin{equation}\label{int:repre:lap}
(-\Delta_{\mathbb{S}^n})^{\alpha}f(x)=c_{n,\alpha}\textrm{P.V.} \int_{\mathbb{S}^n}\frac{f(x)-f(y)}{d(x,y)^{n+2\alpha}}k_{n,\alpha}(d(x,y)) \ d\!\operatorname{vol}(y).
\end{equation}
\end{proposition}
\begin{proof} Let us give a sketch of the proof which follows the arguments in \cite{AOCM3}. Using the semigroup action of the heat kernel one can express \eqref{int:lapla:bel} as 
\[ \int_{0}^{\infty}\!\!\int_{\mathbb{S}^{n}} G(x,y,t)(f(x)-f(y)) \ d\textrm{vol}(y) \dfrac{dt}{t^{1+\alpha}}. \]
Changing the order of integration, one has that
\[\lim_{\epsilon\rightarrow 0}\int_{\mathbb{S}^{n}\setminus B_{\epsilon}(x)}(f(x)-f(y))\int_{0}^{\infty} G(x,y,t) \dfrac{dt}{t^{1+\alpha}} \ d\textrm{vol}(y).\]
Notice that to make this step rigorous, it is necessary to subtract a small ball of radius epsilon. This avoids the singularity and allows to apply Fubini's theorem. This limit is the principal value from the statement. Therefore, it is sufficient to check that there exists a $k_{n,\alpha}(d(x,y))$ such that
\[ \int_{0}^{\infty} G(x,y,t) \dfrac{dt}{t^{1+\alpha}} = \dfrac{k_{n,\alpha}(d(x,y))}{d(x,y)^{n+2\alpha}}. \]
We study the order of the singularity by splitting the integral in two different parts, namely, for small times and for large times. For small times, one needs to use the heat kernel parametrix expansion to check that the singularity is of order $-n-2\alpha$ as stated. Large times can be handled using the explicit expression of the heat kernel via spherical harmonics
\[ G(x,y,t)= \displaystyle\sum_{k=0}^{\infty} e^{-k(k+n)t}C^{\frac{n-2}{2}}_{k}(\cos(d(x,y))). \]
Notice that the Gegenbauer polynomials are eigenfunctions of the Laplacian satisfying proper bounds, \cite{Sogge}. Hence, making use of the exponential decay and the beforementioned estimates of $C^{\frac{n-2}{2}}_{k}$ for $d(x,y)\in(0,\pi)$, one can check that the large times integral is bounded by a function which only depends on the distance $d(x,y)$, $n$ and $\alpha$. Moreover, we can assure that the kernel is positive and decreasing for $d(x,y)\in (0,\pi)$, due to the decreasing  property of the heat kernel on the sphere of Theorem \ref{decreasing}.
\end{proof}

The integral representation \eqref{int:repre:lap} has as a consequence the C\'ordoba-C\'ordoba inequality on the sphere (cf. \cite{CC,CC2}), which is a surprising pointwise inequality for a non local operator.
\begin{corr}
For $\alpha\in(0,1)$ and $f$ smooth enough, the following inequality
\[2 f(x)(-\Delta_{\mathbb{S}^n})^{\alpha}f(x)\geq (-\Delta_{\mathbb{S}^n})^{\alpha}(f^2)(x)\]
holds true at every $x\in\mathbb{S}^{n}$.
\end{corr}

Careful inspection of the proof shows that it also works for arbitrary compact Riemannian manifolds (cf. \cite{CM}). It is also easy to check that it satisfies the following maximum principle.
\begin{corr} 
Let $f$ be a smooth function on the sphere and denote $\bar{x}\in\mathbb{S}^{n}$ the point where it reaches it maximum. Then, for every $\alpha\in(0,1)$,
\[(-\Delta_{\mathbb{S}^n})^{\alpha}f(\bar{x})\geq 0.\]
\end{corr}
Again, this result works equally well for the class of compact Riemannian manifolds. For other results on integral representations of fractional Laplace-Beltrami operators on general compact Riemannian manifolds and further applications to the surface quasi-geostrophic equations, we refer the reader to \cite{AOCM3}.

Let us now depart from this application to another, also related to the subordination technique. Many authors have considered inequalities  involving the Legendre polynomials $P_n$ (see \cite{szego}). One of the simplest and more famous is Fej\'er's inequality
$\sum_{n=0}^N P_n(x)> 0$ for $-1<x<1$ that we still find in recent research \cite{AlKw}. It can be translated into an inequality for the derivative of $P_n$ using the relation
\[
 P_{2K}+P_{2K+1}=1+x+(x^2-1)\sum_{n=1}^{2K}\frac{2n+1}{n(n+1)}P_n'.
\]
When $n$ and $x$ vary, there are  several classic results by Hilb, Stieltjes, and other authors \cite{szego} showing involved oscillations of $P_n$ and $P_n'$ related to the $J_0$ Bessel function. By this reason, it is in general difficult to find inequalities for sums of these polynomials. For $\mathbb{S}^2$, Theorem~\ref{revolution} establishes one of these inequalities and with simple arguments can be presented in a certain general fashion.

Recall that a function $F:\R^+\longrightarrow\R$ is said to be \emph{completely monotonic} if $(-1)^nF^{(n)}>0$ for $n\in\Z_{\ge 0}$. 

\begin{corr}
 Let $F$ be a completely monotonic function such that $F(x)=O\big(x^{-\sigma}\big)$ for some $\sigma>2$. Then,
 \[
  G(x)
  =
  \sum_{n=0}^\infty 
  (2n+1)P_n'(x)F(n(n+1))
 \]
 defines a positive continuous function for $-1<x<1$. 
\end{corr}

Thanks to the identity $(1-x^2)(2n+1)P_n'(x)=n(n+1)\big(P_{n-1}(x)-P_{n+1}(x)\big)$, see \cite[8.914.2]{GrRy},  the previous corollary can also be rephrased in terms of a sum of $P_n$.

\begin{proof}
 It is well known \cite[(7.33.8)]{szego} that $|P_n'(x)|\le P_n'(1)=n(n+1)/2$, and the convergence to a continuous function follows from the Weierstrass $M$-test. On the other hand, Bernstein's theorem \cite[\S IV.12]{widder} assures $F(u)=\int_0^\infty e^{-tu}\; d\sigma(t)$ for some nonnegative measure $d\sigma$.
 Then, for $0<\theta<\pi$,
 \[
  -G(\cos\theta)\sin\theta
  =
  \frac{d}{d\theta}
  \int_0^\infty
  \sum_{n=0}^\infty
  (2n+1)
  P_n(\cos\theta)
  e^{-n(n+1)t}
  \; d\sigma.
 \]
 The sum under the integral is the heat kernel (based on the north pole) for $\mathbb{S}^2$ with the usual coordinates and we know by Theorem~\ref{decreasing} that it decreases with the latitude.
\end{proof}

The function $F(x)= \exp(-x^\alpha)$ is completely monotonic for $0<\alpha\le 1$ and the previous result leads to the following

\begin{corr}\label{fracLB}
 Consider the fractional heat equation $u_t+(-\Delta_{\mathbb{S}^2})^\alpha u=0$ with $0<\alpha<1$ and $(-\Delta_{\mathbb{S}^2})^\alpha$ the (spectral) fractional Laplace-Beltrami operator on~$\mathbb{S}^2$. Then, its fundamental solution decreases as a function of the geodesic distance. 
\end{corr}

This generalizes to other operators if the symbol has a completely monotonic derivative because a simple calculation \cite{merkle} shows that if $F'(x)$ is completely monotonic then $\exp(-tF(x))$ is also completely monotonic for all $t>0$. 

In the noncompact setting an analogous example is the hyperbolic plane $\mathbb{H}^2$. The heat kernel is in this case \cite[(3.37)]{terras}
\[
 \frac{1}{2\pi}
 \int_{0}^\infty
 e^{-(\frac{1}{4}+v^2)t}
 P_{-1/2+iv}(\cosh r) v\tanh(\pi v)\; dv,
\]
where $P_{-1/2+iv}$ is the classical (conical) Legendre function that gives spherical eigenfunctions of the Laplace-Beltrami operator with eigenvalue $1/4+v^2$. 
Proceeding as before, we get the next

\begin{corr}
 Let $F$ be a completely monotonic function such that $F(x)=O\big(x^{-\sigma}\big)$ for some $\sigma>5/4$. Then,
 \[
  G(r)
  =
 \int_{0}^\infty
 F\left(\frac14+v^2\right)
 P'_{-1/2+iv}(\cosh r) v\tanh(\pi v)\; dv
 \]
 defines a positive continuous function for $r>0$. 
\end{corr}
The convergence is assured under the stated conditions thanks to \cite[8.723.1]{GrRy}, with the trivial estimate $O(1)$ for the hypergeometric function, and the fact that $\Gamma(iv)/\Gamma(1/2+iv)\sim v^{-1/2}$  leads to $P'_{-1/2+iv}=O(v^{1/2})$.

The Mehler-Fock transform involving $P_{-1/2+iv}$ is a kind of hyperbolic version of the Fourier transform, and then the previous result has the same flavor as some old results by Bochner and other authors (e.g., \cite{bochner} \cite{herz}).

Using the spectral expansion into spherical functions \cite[Proposition~1.6]{iwaniec}, it is possible to give sense to  the fractional  Laplace-Beltrami  operator on $\mathbb{H}^2$ and conclude the analog of Corollary~\ref{fracLB}.

\section{Final remarks}
If the manifold is not spherically symmetric, one cannot hope a monotone behavior on the geodesic distance. For instance, for the flat torus
$\R^2/(\Z\times L\Z)$
with $L>1$, the heat kernel based on the origin is
\[
 \frac 1L
 \sum_{n,m\in\Z}
 e^{-4\pi^2 t(n^2+m^2/L^2)+2\pi i(nx+my/L)}
 =
 \frac{1}{4\pi t}
 \sum_{n,m\in\Z}
 e^{-\big((n-x)^2+(Lm-y)^2\big)/4t},
\]
where the equality follows from the Poisson summation formula. It is apparent that it decays faster in the second coordinate. On the other hand, in this example we can separate the variables and for $y$ fixed the kernel decreases when $0<x<1/2$ and, in the same way, it decreases when $0<y<L/2$ for $x$ fixed. Any radial derivative is a combination of these, which shows that the heat kernel on a torus is decreasing in any radial direction departing from the origin up to the boundary of the fundamental domain.

It seems natural to try to describe whether there is a region with this monotonicity property in an arbitrary manifold. Notice that, in such a case, the region would depend on the initial point due to the possible lack of symmetries. An inspection of the standard sphere and torus seems to suggest that, for every point, the monotonicity should hold up to its cut locus. This might be a by-product of some wishful thinking. On the other hand the sphere is, in some sense, the worst case scenario because the south pole is heated  through every geodesic from the north pole. 

\section*{Acknowledgments}

This paper emerged from conversations during the long coffee breaks at the Harmonic Analysis in Winter Workshop 2018, held in Madrid. The authors are grateful to its organizers. A. Mas is grateful to M. Cozzi for pointing out the version of the parabolic maximum principle we apply in the proof.

D. Alonso-Or\'an, F. Chamizo and \'A. D. Mart\'inez were partially supported by the {MTM2017-83496-P} project of the MCINN (Spain) {and the ``Severo Ochoa Programme for Centres of Excellence in R{\&}D'' (SEV-2015-0554)}. A. Mas is  supported by
MTM2017-84214 and MTM2017-83499 projects of the MCINN (Spain),
2017-SGR-358 project of the AGAUR (Catalunya) and ERC-2014-ADG project 
HADE Id.\! 669689.

\end{document}